\documentclass{amsart}

\newtheorem{theorem}[subsubsection]{Theorem}
\newtheorem{lemma}[subsubsection]{Lemma}
\newtheorem{proposition}[subsubsection]{Proposition}
\newtheorem{corollary}[subsubsection]{Corollary}

\theoremstyle{definition}
\newtheorem{definition}[subsubsection]{Definition}

\theoremstyle{remark}
\newtheorem{remark}[subsubsection]{Remark}

\DeclareMathAlphabet{\mathpzc}{OT1}{pzc}{m}{it}
\usepackage{amsmath}
\usepackage{amssymb}
\usepackage{euscript}
\usepackage{array}
\usepackage[all]{xy}
\usepackage{amscd}

 \newcommand{\FF}{{\mathbb F}}

 \newcommand{\QQ}{{\mathbb Q}}
 \newcommand{\CC}{{\mathbb C}}
 \newcommand{\NN}{{\mathbb N}}
 
 \newcommand{\GG}{{\mathbb G}}

\renewcommand{\a}{\alpha}   
  \renewcommand{\th}{\theta}
  
  \renewcommand{\th}{\theta}
    
 \newcommand{\s}{\sigma}

 \newcommand{\ord}{\operatorname{ord}}

 \newcommand{\GL}{\operatorname{GL}}

 \newcommand{\Mat}{\operatorname{Mat}}
 \newcommand{\trdeg}{\operatorname{tr.deg}}
 \newcommand{\lcm}{\operatorname{lcm}}
 
 \newcommand{\Plog}{\operatorname{Plog}}
 \newcommand{\Span}{\operatorname{Span}}

\begin{document}

\title[Difference equations and algebraic independence of zeta values]{Frobenius difference equations and algebraic independence of zeta values in positive equal characteristic}

\author{Chieh-Yu Chang}
\address{National Center for Theoretical Sciences, Mathematics Division,
National Tsing Hua University, Hsinchu City 30042, Taiwan
  R.O.C.}
\address{Department of Mathematics, National Central
  University, Chung-Li 32054, Taiwan R.O.C.}

\email{cychang@math.cts.nthu.edu.tw}

\thanks{The first author was supported by NSC and NCTS. The second author was supported by NSF Grant DMS-0600826.
The third author was supported by NSC Grant No.\ 96-2628-M-007-006.}

\author{Matthew A. Papanikolas} \address{Department of Mathematics,
  Texas A${\&}$M University, College Station, TX 77843-3368, USA.}
\email{map@math.tamu.edu}

%    Information for third author
\author{Jing Yu} \address{Department of Mathematics, National Tsing
  Hua University and National Center for Theoretical Sciences, Hsinchu
  City 300, Taiwan R.O.C.}  \email{yu@math.cts.nthu.edu.tw }

%    General info

\date{May 4, 2009}


\begin{abstract}
  In analogy with the Riemann zeta function at positive integers, for
  each finite field ${\FF}_{p^{r}}$ with fixed characteristic $p$ we
  consider Carlitz zeta values $\zeta_{r}(n)$ at positive integers
  $n$. Our theorem asserts that among the zeta values in
  $\cup_{r=1}^{\infty} \{ \zeta_{r}(1),\zeta_{r}(2),\zeta_{r}(3),\dots
  \}$, all the algebraic relations are those relations within each
  individual family $\{ \zeta_{r}(1),\zeta_{r}(2),\zeta_{r}(3),\dots
  \}$. These are the algebraic relations coming from the Euler-Carlitz
  relations and the Frobenius relations. To prove this, a motivic
  method for extracting algebraic independence results from systems of
  Frobenius difference equations is developed.
\end{abstract}

\maketitle

\section{Introduction}
\subsection{Motivic transcendence theory}
Classically Grothendieck's period conjecture for abelian varieties
predicts that the dimension of the Mumford-Tate group of an abelian
variety over $\overline{\QQ}$ should be equal to the transcendence
degree of the field generated by its period matrix over
$\overline{\QQ}$. Conjecturally the Mumford-Tate group is the
motivic Galois group from Tannakian duality, and therefore
Grothendieck's conjecture provides an interpretation of the
algebraic relations among periods in question by way of motivic
Galois groups.

In this paper we are concerned with the algebraic independence of
certain special values over function fields with varying finite
constant fields in positive equal characteristic. In particular, we
are interested in special zeta values. In the positive
characteristic world, there is the concept of $t$-motives introduced
by Anderson \cite{Anderson86}, dual to the concept of $t$-modules.
Based on work of the third author in the 1990's, the structure of
$t$-modules is the key for proving many interesting linear
independence results about special values in this setting (see
\cite{Yu97}). The breakthrough in passing from linear independence
to algebraic independence by way of $t$-motives began with Anderson,
Brownawell and the second author, and in particular with the linear
independence criterion of \cite{ABP} (the so-called ``ABP
criterion'').

By introducing a Tannakian formalism for rigid analytically trivial
pre-$t$-motives and relating it to the Galois theory of Frobenius
difference equations, the second author \cite{Papanikolas} has shown
that the Galois group of a rigid analytically trivial pre-$t$-motive
is isomorphic to its difference Galois group.  Furthermore, the
second author has successfully used the ABP criterion to show that
the transcendence degree of the field generated by the period matrix
of an ABP motive (i.e. the pre-$t$-motive comes from a uniformizable
abelian $t$-module) is equal to the dimension of its Galois group.
More generally, we say that a rigid analytically trivial
pre-$t$-motive has the {\bf{GP}} (Grothendieck period) property if
the transcendence degree of the field generated by its period matrix
is equal to the dimension of its Galois group (for more details of
terminology, see \S 2).

Using a refined version of the ABP criterion proved by the first
author \cite{Chang}, we observe that there are many pre-$t$-motives
which are not ABP motives but which have the {\bf{GP}} property.
This motivates us to introduce a method of uniformizing the
Frobenius twisting operators with respect to different constant
fields for those pre-$t$-motives that have the {\bf{GP}} property.
The pre-$t$-motive we obtain in this way is defined over a larger
constant field, but still has the {\bf{GP}} property (cf. Corollary
\ref{cor of refined
  ABP}). This technique is very useful when dealing with the problem
of determining all the algebraic relations among various special
values of arithmetic interest in a fixed positive characteristic. It
is used in this paper to study special zeta values. For another
application to special arithmetic gamma values, see \cite{CPTY}.

\subsection{Carlitz zeta values}
 Let $p$ be a
prime, and let ${\FF}_{p^{r}}[\th]$ be the polynomial ring in $\th$
over the finite field with $p^r$ elements.  Our aim is to determine
all the algebraic relations among the following zeta values:
\[
\zeta_{r}(n):=\sum_{
\begin{array}{c}
  a\in {\FF}_{p^{r}}[\th] \\
  a\ \mathrm{monic} \\
\end{array}}\frac{1}{a^{n}}\in {\FF}_{p^{r}}(( 1/\th ))\subseteq
\overline{{\FF}_{p}}(( 1/\th )),
\]
where $r$ and $n$ vary over all positive integers.

The study of these zeta values was initiated in $1935$ by Carlitz
\cite{Carlitz1}. For a fixed positive integer $r$, Carlitz
discovered that there is a constant $\tilde{\pi}_{r}$, algebraic
over ${\FF}_{p^{r}}((\frac{1}{\th}))$, such that $\zeta_{r}(n)/
\tilde{\pi}_{r}^{n}$ lies in ${\FF}_{p}(\th)$ if $n$ is divisible by
$p^{r}-1$.  The quantity $\tilde{\pi}_{r}$ arises as a fundamental
period of the Carlitz ${\FF}_{p^{r}}[t]$-module $C_{r}$, and
Wade~\cite{Wade} showed that $\tilde{\pi}_{r}$ is transcendental
over ${\FF}_{p}(\th)$.

We say that a positive integer $n$ is $(p, r)$-{\bf{even}} if it is
a multiple of $p^{r}-1$.  Thus the situation of Carlitz zeta values
at $(p, r)$-{\bf{even}} positive integers is completely analogous to
that of the Riemann zeta function at even positive integers. For
these $(p, r)$-{\bf{even}} $n$, we call the ${\FF}_{p}(\th)$-linear
relations between $\zeta_{r}(n)$ and ${\tilde{\pi}}_{r}^{n}$ the
{\bf{Euler-Carlitz}} relations. Because the characteristic is
positive, there are also {\bf{Frobenius}} $p$-th power relations
among these zeta values: for positive integers $m,n$,
\[
  \zeta_{r}(p^{m}n)=\zeta_{r}(n)^{p^{m}}.
\]

In the 1990's, Anderson and Thakur \cite{Anderson-Thakur} and
Yu~\cite{Yu91}, \cite{Yu97} made several breakthroughs toward
understanding Carlitz zeta values. Using the $t$-module method, the
transcendence of $\zeta_{r}(n)$ for all positive integers $n$, in
particular for ``\textbf{odd}'' $n$ (i.e., $n$ not divisible by
$p^{r}-1$) was proved and it was also proved that the Euler-Carlitz
relations are the only $\overline{{\FF}_{p}(\th)}$-linear relations
among $\{ \zeta_{r}(n),{\tilde{\pi}}_{r}^{m};\hbox{ }m,n\in
{\NN}\}$. Recently, the first and third authors in \cite{Chang-Yu}
used ABP motives instead of $t$-modules to show that for fixed $r$
the Euler-Carlitz relations and the Frobenius $p$-th power relations
account for all the algebraic relations over $\overline{\FF_p(\th)}$
among the Carlitz zeta values
\[
\tilde{\pi}_{r}, \zeta_{r}(1),\zeta_{r}(2),\zeta_{r}(3),\dots.
\]

To complete the story of Carlitz zeta values, the next natural
question is what happens if $r$ varies.   In 1998,
Denis~\cite{Denis98} proved the algebraic independence of all
fundamental periods
$\{\tilde{\pi}_{1},\tilde{\pi}_{2},\tilde{\pi}_{3},\dots \}$ as the
constant field varies. Thus, in view of  \cite{Chang-Yu}, one
expects that for the bigger set of zeta values,
\[
  \cup_{r=1}^{\infty} \{ \zeta_{r}(1),\zeta_{r}(2),\zeta_{r}(3),\dots \},
\]
the {\bf{Euler-Carlitz}} relations and the {\bf{Frobenius}} $p$-th
power relations still account for all the algebraic relations. This
is indeed the case as we find from the following theorem (stated
subsequently as Corollary~\ref{cor of the main thm}).

\begin{theorem}\label{main thm introduction}
 Given any positive integers $s$ and $d$, the transcendence degree of the
field
$$\overline{{\FF}_{p}(\th)}(\cup_{r=1}^{d}\{
\tilde{\pi}_{r},\zeta_{r}(1), \dots, \zeta_{r}(s)\}) $$
 over $\overline{{\FF}_{p}(\th)}$ is
\[
\sum_{r=1}^{d} \biggl( s-\biggl\lfloor\frac{s}{p} \biggr\rfloor -
\biggl\lfloor \frac{s}{p^{r}-1}\biggr\rfloor+ \biggl\lfloor
\frac{s}{p(p^{r}-1)} \biggr\rfloor+1  \biggr).
\]
\end{theorem}

\subsection{Outline}

Our strategy is to construct a pre-$t$-motive which has the
{\bf{GP}} property and whose period matrix accounts for the Carlitz
zeta values in question. In \cite{Chang-Yu} an ABP motive has
already been constructed for Carlitz zeta values with respect to a
fixed constant field. The problem here is one concerning varying the
constant fields in a fixed characteristic, and one has to uniformize
Frobenius powers in order to apply the method developed in
\cite{Papanikolas}.

This paper is organized as follows. In \S 2, we review Papanikolas'
theory and investigate the pre-$t$-motives that have the {\bf{GP}}
property.  Here we introduce the mechanism of uniformizing Frobenius
twisting operators while taking direct sums. Section 3 includes
discussions about rigid analytically trivial pre-$t$-motives of type
{\bf{SV}}, i.e., their Galois groups are extensions of split tori by
vector groups. The heart of this section is Theorem \ref{main thm}
where we determine the dimensions of Galois groups of direct sums of
pre-$t$-motives of type {\bf{SV}}. The pre-$t$-motive for Theorem
\ref{main thm introduction} is constructed in \S 4 and we prove that
it satisfies the conditions of Theorem \ref{main thm}. Finally, we
calculate its dimension explicitly in Theorem \ref{main thm of last
  sec}, which then has Theorem \ref{main thm introduction} as direct
consequence.

\section{$t$-motivic Galois groups}
\subsection{Notation.}
\subsubsection{Table of symbols.}${}$
\\$\mathbb{F}_{p}:=$ the finite field of $p$ elements, $p$ a
prime number.
\\$k:=\mathbb{F}_{p}(\theta):=$ the rational function field in the
variable $\th$ over ${\FF}_{p}$.
\\$k_{\infty}:=\mathbb{F}_{p}((\frac{1}{\theta}))$, completion of
$k$ with respect to the infinite place.
\\$\overline{k_{\infty}}:=$ a fixed algebraic closure of $k_{\infty}$.
\\$\overline{k}:=$ the algebraic closure of $k$ in $\overline{k_{\infty}}$.
\\$\mathbb{C}_{\infty}:=$ completion of $\overline{k_{\infty}}$ with
respect to the canonical extension of the infinite place.
\\$|\cdot|_{\infty}:=$ a fixed absolute value for the completed
field $\mathbb{C}_{\infty}$.
\\$\mathbb{T}:=\{f\in \mathbb{C}_{\infty}[[t]] \mid
\textnormal{$f$ converges on $|t|_{\infty}\leq_{}1$} \}$. This is
known as the Tate algebra.
\\$\mathbb{L}:=$ the fraction field of $\mathbb{T}$.
\\${\GG}_{a}:=$ the additive group.
\\$\GL_r/F:=$ for a field $F$, the $F$-group scheme of invertible $r\times r$ square matrices.
\\${\GG}_{m}:=\GL_{1}$, the multiplicative group.

\subsubsection{Block diagonal matrices} Let $A_{i}\in
\Mat_{m_{i}}(\mathbb{L})$ for $i=1, \dots, n$, and
$m:=m_{1}+\cdots+m_{n}$. We define $\oplus_{i=1}^{n} A_{i} \in
\Mat_m(\mathbb{L})$ to be the canonical block diagonal matrix, i.e.,
the matrix with $A_1, \dots, A_n$ down the diagonal and zeros
elsewhere.
%$$\oplus_{i=1}^{n}
%A_{m_{i}}:= \left[%
%\begin{array}{ccc}
%  A_{m_{1}} & \cdots & 0\\
%  \vdots & \ddots &  \vdots\\
%  0 & \cdots & A_{m_{n}} \\
%\end{array}%
%\right]\in\Mat_{m}(\mathbb{L}).  $$

\subsubsection{Twisting} For $q$ a fixed power of $p$, we let
$\mathbb{F}_q$ denote the field with $q$ elements, and we define an
operator $\s=\s_{q}$ on $\CC_{\infty}$ by $x\mapsto
x^{\frac{1}{q}}$. Then we extend this operator to
$\CC_{\infty}((t))$ as follows. For $n\in \mathbb{Z}$ and a formal
Laurent series $f=\sum_{i} a_{i}t^{i}\in \mathbb{C}_{\infty}((t))$
we define ${\s}^{n}(f):=f^{(-n)}:=\sum_{i}a_{i}^{q^{-n}}t^{i}$. The
operation is an automorphism of the Laurent series field
$\mathbb{C}_{\infty}((t))$ that stabilizes several subrings, e.g.,
$\overline{k}[[t]]$, $\overline{k}[t]$, and $\mathbb{T}$. More
generally, for any matrix $B$ with entries in
$\mathbb{C}_{\infty}((t))$ we define $B^{(-n)}$ by the rule
${B^{(-n)}}_{ij}={B_{ij}}^{(-n)}$.
%In particular, for any matrix $B$
%with coefficients in $\mathbb{C}_{\infty}$ we have
%$\bigl({B^{(-n)}}\bigr)_{ij}=\bigl(B_{ij}\bigr)^{q^{-n}}$.

\subsubsection{Entire power series}
A power series $f=\sum_{i=0}^{\infty}a_{i}t^{i}\in
\mathbb{C}_{\infty}[[t]]$ that satisfies
\[
  \lim_{i\rightarrow \infty}\sqrt[i]{|a_{i}|_{\infty}}=0
\]
and
\[
  [k_{\infty}(a_{0},a_{1},a_{2},\dots):k_{\infty}]< \infty
\]
is called an entire power series. As a function of $t$, such a power
series $f$ converges on all $\mathbb{C}_{\infty}$ and, when
restricted to $\overline{k_{\infty}}$, $f$ takes values in
$\overline{k_{\infty}}$. The ring of the entire power series is
denoted by $\mathbb{E}$.

\subsection{Pre-$t$-motives and the {\bf{GP}} property}
Let $\bar{k}(t)[{\s},{\s}^{-1}]$ be the noncommutative ring of
Laurent polynomials in ${\s}$ with coefficients in $\bar{k}(t)$,
subject to the relation
\[
  {\s}f:=f^{(-1)} {\s} \hbox{ for all }f\in \bar{k}(t).
\]
A pre-$t$-motive $M$ over $\FF_{q}$ is a left
$\bar{k}(t)[\s,{\s}^{-1}]$-module that is finite dimensional over
$\bar{k}(t)$. Letting $\mathbf{m}\in\Mat_{r\times 1}(M)$ be a
$\bar{k}(t)$-basis of $M$, multiplication by $\s$ on $M$ is
represented by
\[
  \s (\mathbf{m})=\Phi \mathbf{m}
\]
for some matrix $\Phi \in \GL_{r}(\bar{k}(t))$. Furthermore, $M$ is
called rigid analytically trivial if there exists $\Psi\in
\GL_{r}(\mathbb{L})$ such that
\[
  \s(\Psi):=\Psi^{(-1)}=\Phi \Psi.
\]
Such a matrix $\Psi$ is called a rigid analytic trivialization of
the matrix $\Phi$. We also say that $\Psi$ is a rigid analytic
trivialization of $M$ (with respect to $\mathbf{m}$). Note that if
$\Psi'\in \GL_{r}(\mathbb{L})$ is also a rigid analytic
trivialization of $\Phi$, then ${\Psi'}^{-1}\Psi \in
\GL_{r}(\FF_{q}(t))$ (cf. {\cite[\S 4.1.6]{Papanikolas}}). Moreover,
if we put $\mathbf{m}':=B \mathbf{m}$ for any fixed $B\in
\GL_{r}(\bar{k}(t))$, then $\Phi':=B^{(-1)}\Phi B^{-1}$ represents
multiplication by $\s$ on $M$ with respect to the $\bar{k}(t)$-basis
$\mathbf{m}'$ of $M$ and $\Psi':=B \Psi$ is a rigid analytic
trivialization of $\Phi'$.

\begin{definition}
  Suppose we are given a rigid analytically trivial pre-$t$-motive $M$
  over $\FF_{q}$ that is of dimension $r$ over $\bar{k}(t)$. If there
  exists a $\bar{k}(t)$-basis $\mathbf{m}\in \Mat_{r \times
    1}(M)$ so that there exists $\Psi\in \GL_{r}(\mathbb{L})\cap
  \Mat_{r}(\mathbb{E})$ which is a rigid analytic trivialization
  of $M$ with respect to $\mathbf{m}$ and satisfies
\[
\trdeg_{\bar{k}(t)}\bar{k}(t)(\Psi) =
\trdeg_{\bar{k}}\bar{k}(\Psi(\th)),
\]
then we say that $M$ has the {\bf{GP}} (Grothendieck period)
property, where $\bar{k}(t)(\Psi)$ (resp. $\bar{k}(\Psi(\th))$) is
the field generated by all entries of $\Psi$ (resp. $\Psi(\th)$)
over $\bar{k}(t)$ (resp. $\bar{k}$).  The {\bf{GP}} property is
independent of the choices of $\Psi$ for a fixed $\mathbf{m}$.
\end{definition}

For any $r\in \NN$, we let $\FF_{q^{r}}$ be the finite field of
$q^{r}$ elements. Given a rigid analytically trivial pre-$t$-motive
$M$ over $\FF_{q}$ with ($\mathbf{m}$, $\Phi$, $\Psi$) as above we
define its $r$-{th} derived pre-$t$-motive $M^{(r)}$ over
$\FF_{q^{r}}$ (with respect to the operator ${\s}^{r}$): the
underlying space of $M^{(r)}$ is the same as $M$, but it is now
regarded as a left $\bar{k}(t)[\s^{r},\s^{-r}]$-module. Letting
\[
\Phi':=\Phi^{(-(r-1))} \cdots \Phi^{(-1)} \Phi,
\]
we have ${\s}^{r}\mathbf{m}=\Phi' \mathbf{m}$ and $\s^{r}\Psi:=
\Psi^{(-r)}=\Phi' \Psi$, and hence $\Psi$ is also a rigid analytic
trivialization of $M^{(r)}$.

\begin{proposition}\label{direct sum of GP motives}
Let $M$ be a rigid analytically trivial pre-$t$-motive over
$\FF_{q}$ which has the {\bf{GP}} property. For any positive integer
$r$, the $r$-th derived pre-$t$-motive $M^{(r)}$ over $\FF_{q^{r}}$
of $M$ is also rigid analytically trivial and has the {\bf{GP}}
property.
\end{proposition}

Following the work of \cite{Papanikolas}, one can show:
\begin{theorem}{\rm{(Chang \cite[Thm.~1.2]{Chang},
Papanikolas \cite[Thm.~5.2.2]{Papanikolas})}} \label{refined ABP}
Suppose $\Phi\in \Mat_{r}(\bar{k}[t])$ defines a rigid analytically
trivial pre-$t$-motive $M$ over $\FF_{q}$  with a rigid analytic
trivialization $\Psi\in \Mat_{r}(\mathbb{T})\cap
\GL_{r}(\mathbb{L})$. If $\det \Phi(0)\neq 0$ and $\det
\Phi(\th^{\frac{1}{q^{i}}})\neq 0$ for all $i=1,2,3,\ldots$, then
$M$ has the {\bf{GP}} property.
\end{theorem}

Note that by {\cite[Prop.~3.1.3]{ABP}} the condition $\det
\Phi(0)\neq 0$ implies $\Psi\in \Mat_{r}(\mathbb{E})$. Combining
Theorem \ref{refined ABP} and Proposition \ref{direct sum of GP
motives} we have:

\begin{corollary}\label{cor of refined ABP}
  Given an integer $d \geq 2$, we let $\ell:=\lcm(1,\ldots,d)$. For
  each $1\leq i \leq d$, let $\ell_{i}:=\frac{\ell}{i}$ and let
  $\Phi_{i}\in \Mat_{r_{i}}(\bar{k}[t])\cap \GL_{r_{i}}(\bar{k}(t))$
  define a pre-$t$-motive $M_{i}$ over $\FF_{q^{i}}$ with a rigid
  analytic trivialization $\Psi_{i}\in \Mat_{r_{i}}(\mathbb{T})\cap
  \GL_{r_{i}}(\mathbb{L})$.  Suppose that each $\Phi_{i}$ satisfies
  the hypotheses of Theorem \ref{refined ABP} for $i=1,\ldots,d$.
  Then the direct sum
\[
M:= \oplus_{i=1}^{d} {M_{i}}^{(\ell_{i})}
\]
is a rigid analytically trivial pre-$t$-motive over $\FF_{q^{\ell}}$
that has the {\bf{GP}} property.
\end{corollary}

\begin{proof}
For each $1\leq i \leq d$, we define
$$ \Phi_{i}':=\Phi_{i}^{(-(\ell_{i}-1))}\cdots\Phi_{i}^{(-i)}\Phi_{i} .  $$
Moreover, if we define
$$  \Phi':=\oplus_{i=1}^{d} \Phi_{i}',\quad \Psi':=\oplus_{i=1}^{d} \Psi_{i}, $$
then we have $$\Psi'^{(-\ell)}=\Phi' \Psi'.   $$ Note that the
matrix representing multiplication by $\s^{\ell}$ on $M$ with
respect to the evident $\bar{k}(t)$-basis is given by  $\Phi'$.

Our task is to show that $\Phi'$ satisfies the hypotheses of Theorem
\ref{refined ABP} (with respect to the operator $\s^{\ell}$), whence
the result. It is obvious that $\det \Phi'(0) \neq 0 $ since $\det
\Phi_{i}(0)\neq 0$ for each $1\leq i \leq d$. Suppose that $\det
\Phi' (\th^{\frac{1}{q^{\ell j}}} )=0 $ for some $j\in \NN$. This
implies that there exists $1\leq i \leq d$ and $0\leq m \leq
\ell_{i}-1$ so that
\[
\det\Phi_{i}^{(-im)}(\th^{(-\ell j ) }) =0.
\]
However, this is equivalent to
\begin{equation}\label{vanishing of det Phi i}
 \det \Phi_{i}(\th^{( -( \ell j -im)   )  }  ) =0.
\end{equation}
Since $0\leq m\leq \ell_{i}-1 $ and $i| \ell$, we have that $(\ell
j-im)>0 $ and $i | ( \ell j-im )  $. Thus, (\ref{vanishing of det
Phi i}) contradicts the hypothesis that $\det \Phi_{i}(\th ^{(-i h)}
)\neq 0\hbox{ for all }h=1,2,3,\ldots$.
\end{proof}

\subsection{Difference Galois groups and transcendence}
In this section, we review the related theory developed in
\cite{Papanikolas}. The category of pre-$t$-motives over $\FF_{q}$
forms a rigid abelian ${\FF}_{q}(t)$-linear tensor category.
Moreover, the category $\mathcal{R}$ of rigid analytically trivial
pre-$t$-motives over $\FF_{q}$ forms a neutral Tannakian category
over ${\FF}_{q}(t)$. Given an object $M$ in $\mathcal{R}$, we let
$\mathcal{T}_{M}$ be the strictly full Tannakian subcategory of
$\mathcal{R}$ generated by $M$. That is, $\mathcal{T}_{M}$ consists
of all objects of $\mathcal{R}$ isomorphic to subquotients of finite
direct sums of
\[
M^{\otimes u}\otimes (M^{\vee})^{\otimes v} \hbox{ for various }u,
v,
\]
where $M^{\vee}$ is the dual of $M$.  By Tannakian duality,
$\mathcal{T}_{M}$ is representable by an affine algebraic group
scheme $\Gamma_{M}$ over ${\FF}_{q}(t)$. The group $\Gamma_{M}$ is
called the Galois group of $M$ and it is described explicitly as
follows.

Suppose that $\Phi\in \GL_{r}(\bar{k}(t))$ provides multiplication
by $\s$ on $M$ with respect to a fixed basis $\mathbf{m}\in
\Mat_{r\times 1}(M)$ over $\bar{k}(t)$. Let $\Psi\in
\GL_{r}(\mathbb{L})$ be a rigid analytic trivialization for $\Phi$.
Let $X:=(X_{ij})$ be an $r\times r$ matrix whose entries are
independent variables $X_{ij}$, and define a
$\overline{k}(t)$-algebra homomorphism $\nu :
\overline{k}(t)[X,1/\det X] \to \mathbb{L}$ so that $\nu(X_{ij}) =
\Psi_{ij}$.  We let
\begin{align*}
\Sigma_\Psi &:= \mathrm{im}\ \nu = \overline{k}(t)[\Psi,1/\det \Psi]
\subseteq \mathbb{L}, \\
Z_\Psi &:= \mathrm{Spec}\ \Sigma_{\Psi}.
\end{align*}
Then $Z_\Psi$ is a closed $\overline{k}(t)$-subscheme of
$\mathrm{GL}_{r}/\overline{k}(t)$.  Let $\Psi_1$, $\Psi_2 \in
\mathrm{GL}_r(\mathbb{L} \otimes_{\overline{k}(t)} \mathbb{L})$ be
the matrices satisfying $(\Psi_1)_{ij} = \Psi_{ij} \otimes 1$ and
$(\Psi_2)_{ij} = 1 \otimes \Psi_{ij}$.  Let $\widetilde{\Psi} :=
\Psi_1^{-1}\Psi_2$.  We have an $\FF_{q}(t)$-algebra homomorphism
$\mu : \FF_{q}(t)[X,1/\det X] \to \mathbb{L}
\otimes_{\overline{k}(t)} \mathbb{L}$ so that $\mu(X_{ij}) =
\widetilde{\Psi}_{ij}$. Furthermore, we define
\begin{equation} \label{Gamma Psi}
\begin{aligned}
\Delta &:= \mathrm{im}\ \mu, \\
\Gamma_\Psi &:= \mathrm{Spec}\ \Delta.
\end{aligned}
\end{equation}
The following theorem is proved in \cite{Papanikolas}.

\begin{theorem} \label{Galois theory} {\rm (Papanikolas
    \cite[Thm.~4.2.11, 4.3.1, 4.5.10]{Papanikolas})} The scheme $\Gamma_\Psi$ is a closed
  $\FF_{q}(t)$-subgroup scheme of $\mathrm{GL}_{r} /
  \FF_{q}(t)$, which is isomorphic to the Galois group
  $\Gamma_M$ over $\FF_{q}(t)$.  Moreover $\Gamma_\Psi$ has the
  following properties:
\begin{enumerate}
\item[(a)] $\Gamma_\Psi$ is smooth over $\overline{\FF_{q}(t)}$ and
is
  geometrically connected.
\item[(b)] $\dim \Gamma_\Psi = \mathrm{tr.deg}_{\overline{k}(t)}\
  \overline{k}(t)(\Psi)$.
\item[(c)] $Z_\Psi$ is a $\Gamma_\Psi$-torsor over $\overline{k}(t)$.
\end{enumerate}
In particular, if $M$ has the {\bf{GP}} property, then one has
\begin{enumerate}
\item[(d)] $\dim \Gamma_{\Psi}=\mathrm{tr.deg}_{\bar{k}}\
  \bar{k}(\Psi(\th))$.
\end{enumerate}
\end{theorem}

We call $\Gamma_{\Psi}$ the Galois group associated to the
difference equation $\Psi^{(-\ell)}=\Phi \Psi$. This $\Gamma_{\Psi}$
is independent of the analytic trivialization $\Psi$, up to
isomorphism over ${\FF}_{q}(t)$,   Throughout this paper we always
identify $\Gamma_{M}$ with $\Gamma_{\Psi}$, and
 regard it as a linear algebraic
group over ${\FF}_{q}(t)$ because of Theorem \ref{Galois theory}(a).

\begin{remark}\label{remark surjection as projection}
  Let $r_{1},r_{2}$ be positive integers and ${\bf{0}}:=0_{r_{1}\times
    r_{2}}$ the zero matrix of size ${r_1}\times{r_{2}}$. Suppose
  that the matrix $$\Phi:= \left[%
\begin{array}{cc}
  \Phi_{1} & {\bf{0}} \\
  \Phi_{3} & \Phi_{2} \\
\end{array}
\right] \in \GL_{r_{1}+r_{2}}(\bar{k}(t))$$ defines a rigid
analytically trivial pre-$t$-motive $M$. Then one can always find
its rigid analytic trivialization of the form
$$\Psi:=\left[
\begin{array}{cc}
  \Psi_{1} & {\bf{0}} \\
  \Psi_{3} & \Psi_{2} \\
\end{array}
\right]\in \GL_{r_{1}+r_{2}}(\mathbb{L}).$$ By \eqref{Gamma Psi}, we
have that $$ \Gamma_{\Psi}\subseteq \left\{\left[
\begin{array}{cc}
  * & {\bf{0}} \\
  * & * \\
\end{array}
\right] \right\}\subseteq {\mathrm{GL}_{r_{1}+r_{2}/ \FF_{q}(t)}}.$$
Let $N$ be the  sub-pre-$t$-motive of $M$ defined by
$\Phi_{1}\in\GL_{r_{1}}(\bar{k}(t))$ with rigid analytic
trivialization $\Psi_{1}$, then by the Tannakian theory we have a
natural surjective morphism
\begin{equation}\label{surjection as projection}
\begin{array}{cccc}
  \pi: & \Gamma_{\Psi}(\overline{\FF_{q}(t)}) & \twoheadrightarrow & \Gamma_{\Psi_{1}}(\overline{\FF_{q}(t)}). \\
  & \gamma & \mapsto & \pi(\gamma) \\
\end{array}
\end{equation}
In fact, $\pi(\gamma)$ comes from the restriction of the action of
$\gamma$ to the fiber functor of $\mathcal{T}_{N}$ (which is a full
subcategory of $\mathcal{T}_{M}$). Precisely, $\pi(\gamma)$ is the
matrix cut out from the upper left square of $\gamma$ with size
$r_{1}$ (for detailed arguments, see \cite[\S 6.2.2]{Papanikolas}).
\end{remark}

\section{A dimension criterion}
\subsection{Pre-$t$-motives of type {\bf{SV}} }\label{Ai,Fi}

Let $\{n_{1},\ldots,n_{h} \}$ be $h$ non-negative integers. We say
that a  pre-$t$-motive $M$ over $\FF_{q}$ is of type {\bf{SV}} (its
Galois group being an extension of a split torus by a vector group)
if it is defined by $\Phi$ of the form
$$\Phi:=\oplus_{i=1}^{h} A_{i} ,\quad
  A_{i}:=\left[%
\begin{array}{cccc}
  a_{i} & 0 & \cdots & 0 \\
  a_{i1} & 1 & \cdots & 0 \\
  \vdots & \vdots & \ddots & \vdots \\
  a_{in_{i}} & 0 & \cdots & 1 \\
\end{array}%
\right]\in \GL_{1+n_{i}}(\bar{k}(t)),$$ which has rigid analytically
trivialization
$$\Psi:=\oplus_{i=1}^{h} F_{i},\quad F_{i}:=\left[%
\begin{array}{cccc}
  f_{i} & 0 & \cdots & 0 \\
  f_{i1} & 1 & \cdots & 0 \\
  \vdots & \vdots & \ddots & \vdots \\
  f_{in_{i}} & 0 & \cdots & 1 \\
\end{array}%
\right]\in \GL_{1+n_{i}}(\mathbb{L}).$$ The  entries $f_i,
i=1,\dots, h$, shall be called the  diagonals of the rigid analytic
trivialization $\Psi$.

Let $T$ be the Galois group associated to the difference equation
$$ \left[%
\begin{array}{ccc}
  f_{1} & \cdots & 0 \\
  \vdots & \ddots & \vdots \\
   0& \cdots & f_{h} \\
\end{array}%
\right]^{(-1)}=\left[%
\begin{array}{ccc}
  a_{1} &\cdots  & 0 \\
   \vdots& \ddots & \vdots \\
  0 & \cdots & a_{h} \\
\end{array}%
\right] \left[%
\begin{array}{ccc}
  f_{1} & \cdots & 0 \\
  \vdots & \ddots &\vdots  \\
  0 & \cdots & f_{h} \\
\end{array}%
\right] $$ and note that by \eqref{Gamma Psi}, $T$ is a subtorus of
the $h$ dimensional split torus in $\GL_{h}/ \FF_{q}(t)$. By the
same reason as \eqref{surjection as projection}, we have the
following natural projection of  Galois groups
\begin{equation}\label{projection to torus}
\Gamma_{\Psi}\twoheadrightarrow T
\end{equation} given in terms of coordinates by $$ \oplus_{i=1}^{h} \left[%
\begin{array}{cccc}
  x_{i} & 0 & \cdots & 0 \\
  x_{i1} & 1 & \cdots & 0 \\
  \vdots & \vdots & \ddots & \vdots \\
  x_{in_{i}} & 0 & \cdots & 1 \\
\end{array}%
\right] \mapsto [x_{1}]\oplus \dots\oplus [x_{h}].$$ One has also
the following exact sequence of linear algebraic groups
\begin{equation}\label{S.E.S.}
1\rightarrow V \rightarrow \Gamma_{\Psi} \twoheadrightarrow T
\rightarrow 1 ,
\end{equation}
where $V$ is a vector group contained in the
$(\sum_{i=1}^{h}n_{i})$-dimensional \lq\lq coordinate\rq\rq vector
group $G$ which contains all elements of the form
$$\oplus_{i=1}^{h}
\left[%
\begin{array}{cccc}
  1 & 0 & \cdots & 0 \\
  x_{i1} & 1 & \cdots & 0 \\
  \vdots & \vdots & \ddots & \vdots \\
  x_{in_{i}} & 0 & \cdots & 1 \\
\end{array}%
\right].
$$
This subgroup  $V$ is the unipotent radical of its (solvable) Galois
group $\Gamma_\Psi$, and $\dim \Gamma_{\Psi}=\dim V + \dim T$.
\begin{definition}
Let $M$ be a rigid analytically trivial pre-$t$-motive of {\bf SV}
type given as above. We say that its Galois group $\Gamma_{\Psi}$ is
\emph{full} if $\dim V=\sum_{i=1}^{h}n_{i}$, i.e., $V=G$.
\end{definition}

\subsection{Criterion for direct sum motives to have full Galois group}
We continue with the notation of \S\ref{Ai,Fi}.  For each $i$,
$1\leq i\leq h$, the $1\times 1$ matrix $[a_{i}]$ defines a
sub-pre-$t$-motive of $M$ over $\FF_{q}$, which is one-dimensional
over $\bar{k}(t)$. Its rigid analytic trivialization is given by the
diagonal $f_{i}$ satisfying $[f_{i}]^{(-1)}=[a_{i}][f_{i}]$. We call
one such sub-pre-$t$-motive a \emph{diagonal} of the pre-$t$-motive
$M$. Note that by Theorem \ref{Galois theory} the Galois group of a
diagonal is $\GG_{m}$ if and only if the corresponding
trivialization $f_{i}$ is transcendental over $\bar{k}(t)$. In this
situation, the canonical projection $T\rightarrow \GG_{m}$ on the
$i$th coordinate of $T$ is surjective.

Note that there is a canonical action of $T$ on $G$, which induces
an action of $T$ on $V$ compatible with the one coming from
\eqref{S.E.S.}. Hence, we have that given any element $\gamma\in V$
whose $x_{ij}$-coordinate is nonzero, the orbit inside $V$ given by
the action of $T$ on $\gamma$ must be infinite if $f_{i}$ is
transcendental over $\bar{k}(t)$.

\begin{definition}
Let $M_{r}$ be a rigid analytically trivial pre-$t$-motive of type
{\rm{\bf{SV}}} for $r=1,\ldots,d$. We say that this set of
pre-$t$-motives $\{ M_{r} \}_{r=1}^{d} $ is diagonally independent
if for any $1\leq i,j\leq d$, $i\neq j$, the Galois group of
$N_{i}\oplus N_{j}$ is a two-dimensional torus over $\FF_{q}(t)$,
where $N_{i}$ (resp. $N_{j}$) is any diagonal of $M_{i}$ (resp.
$M_{j}$).

\end{definition}

\begin{theorem}\label{main thm}
Let $M_{1},\ldots,M_{d}$ be rigid analytically trivial
pre-$t$-motives over $\FF_{q}$ of type {\rm{\bf{SV}}} with full
Galois groups. Suppose that the set $\{ M_{r} \}_{r=1}^{d}$ is
diagonally independent. Putting $M:=\oplus_{r=1}^{d}M_{r}$, then the
Galois group of $M$ is also full.
\end{theorem}

\begin{proof}
  Without loss of generality, we may assume $d=2$ since the following
  argument generalizes easily for arbitrary $d\geq 2$.  Let the
  pre-$t$-motive $M_i$ be defined by a matrix $\Phi_i$, for $i=1, 2$,
  with $\Psi=\Psi_i\oplus \Psi_2$ a rigid analytic trivialization of
  $M$, and let $\Gamma_{\Psi}$, $\Gamma_{\Psi_1}$, $\Gamma_{\Psi_2}$
  be the Galois groups. The unipotent radicals of these groups are
  denoted by $V$, $V_1$, $V_2$ respectively. Let $G$ (resp.\ $G_i$,
  $i=1, 2$) be the coordinate vector group containing $V$ (resp.\
  $V_i$).  Suppose the matrix $\Phi_1$ has $h$ diagonal blocks, and
  let the coordinates of $G_1$ be denoted by $x_{ij}$, $i=1, \ldots,
  h$, $j=1, \dots, n_i$. Similarly let $y_{ij}$, $i=1, \ldots, \ell$,
  $j=1, \dots, m_\ell$ denote the coordinates of $G_2$. Any subspace
  $W\subseteq G$ obtained by setting some of these coordinates to $0$ is
  called a linear coordinate subspace. The hypothesis that the Galois
  group of $M_i$ is full means exactly that $G_i=V_i$, for $i=1,
  2$. We are going to prove that $G=V$.

  Suppose $V$ has codimension $r$ in the coordinate vector group
  $G$. We can find linear coordinate subspace $W\subseteq G$ of
  dimension $r$ such that $W\cap V$ is of dimension $0$.  Since $W\cap
  V$ is invariant under $T$, it must be $\{0\}$, because the hypothesis
  that $M_1$ and $M_2$ are diagonally independent implies in
  particular that all diagonals in the analytic trivialization $\Psi$
  are transcendental over $\bar{k}(t)$.

  Let $W'\subseteq G$ be the linear coordinate space given by those
  coordinates disjoint from those of $W$. Then the natural projection from
  $G$ to $W'$ induces on $V$ an isomorphism of vector groups.
  Composing the inverse of this isomorphism with the surjective
  morphism $\pi_{1}$ in the following diagram
$$\CD 1@>>>V @>>>\Gamma_{\Psi} @>>> T @>>>1\\
@.     @VV{\pi_{1}}V  @VV{\pi_{1}}V @VVV\\
1@>>>G_1 @>>>\Gamma_{\Psi_1}  @>>> T_1 @>>>1\endCD$$ we obtain a
morphism $\pi_1$ from $W'$ onto $G_1$ which is furthermore a
$T$-morphism.

We contend that under the hypothesis that $M_1$ and $M_2$ are
diagonally independent, $\pi_1$ maps $G_2\cap W'$ to zero. This
contention results from the following basic lemma by taking any
diagonal $N_{1}$ (resp. $N_{2}$) of $M_1$ (resp. $M_2$) and consider
the restriction of the above morphism $\pi_1$ to a single block.
Now since $G=G_1\times G_2$, it follows that $\pi_1(G_1\cap W')=
G_1$. Thus $G_1\subseteq W'$. Similarly we also have $G_2\subseteq
W'$, and hence $G=W'$ and $G=V$.
\end{proof}

\begin{lemma}
Let $G_{1}$ {\rm{(resp.}} $G_{2}${\rm{)}} be the vector group with
coordinates
$$\left[%
\begin{array}{cccc}
  1 & 0 & \cdots & 0 \\
  x_{1} & 1 & \cdots & 0 \\
  \vdots & \vdots & \ddots & \vdots \\
  x_{n} & 0 & \cdots & 1 \\
\end{array}%
\right],\quad
\left(\text{resp.}\ \left[%
\begin{array}{cccc}
  1 & 0 & \cdots & 0 \\
  y_{1} & 1 & \cdots & 0 \\
  \vdots & \vdots & \ddots & \vdots \\
  y_{m} & 0 & \cdots & 1 \\
\end{array}%
\right]\right).
$$
Let $\GG_m^2$ act on $G_{1}$ (resp.\ $G_{2}$)  by
$$\GG_m^2\owns(x, y) : x_i\mapsto x x_i , 1\le i\le n,  \quad
(\text{resp.}\ y_j\mapsto y y_j, 1\le j\le m).$$ If $\pi_1 :
G_{2}\rightarrow G_{1}$ is a $\GG_m^2$-morphism, then $\pi_1\equiv
0$.\end{lemma}

\section{Application to zeta values} \label{AppZeta}

\subsection{Carlitz theory}
Throughout \S\ref{AppZeta} we fix $q:=p$ and $\s=\s_{p}:\sum
a_{i}t^{i}\mapsto \sum a_{i}^{\frac{1}{p}}t^{i}$ for $\sum
a_{i}t^{i}\in \CC_{\infty}((t))$.  Also, we henceforth let $r$ be a
fixed positive integer. Recall the Carlitz
${\FF}_{p^{r}}[t]$-module, denoted by $C_{r}$, which is given by the
following ${\FF}_{p^{r}}$-linear ring homomorphism:
\[
  C_{r} = (t\mapsto (x\mapsto \th x+x^{p^{r}})) :  {\FF}_{p^{r}}[t] \rightarrow
 \mathrm{End}_{{\FF}_{p^{r}}}({\GG}_{a}).
\]
Note that when we regard $C_{r}$ as a Drinfeld
${\FF}_{p}[t]$-module, it is of rank $r$ (see \cite{Goss},
\cite{Thakur}). One has the Carlitz exponential associated to
$C_{r}$:
\[
\exp_{C_{r}}(z):=\sum_{i=0}^{\infty} \frac{z^{p^{ri}}}{D_{ri}}.
\]
Here we set
\[
\begin{array}{rl}
  D_{r0} &:=1,  \\
  D_{ri} &:=\prod_{j=0}^{i-1}(\th^{p^{ri}}-\th^{p^{rj}}  ), \quad i\geq 1.  \\
\end{array}
\]
Now $\exp_{C_{r}}(z)$ is an entire power series in $z$ satisfying
the functional equation
$$ \exp_{C_{r}}(\th z)=\th \exp_{C_{r}}(z)+\exp_{C_{r}}(z)^{p^{r}} .$$
Moreover one has the product expansion
\[
\exp_{C_{r}}(z)=z \prod_{0 \neq a \in {\FF}_{p^{r}}[\th]} \biggl(
1-\frac{z}{a \tilde{\pi}_{r}}  \biggr),
\]
where
\[
  \tilde{\pi}_{r}=\th (- \th)^{\frac{1}{p^{r}-1}}
  \prod_{i=1}^{\infty}\Bigl(1-\th^{1-p^{ri} }\Bigr)^{-1}
\]
is a fundamental period of $C_{r}$. Throughout this paper we will
fix a choice of $(-\th )^{\frac{1}{p^{r}-1}}$ so that
$\tilde{\pi}_{r}$ is well-defined element in
$\overline{{\FF}_{p}((\frac{1}{\th} )) }$.  We also choose these
roots in a compatible way so that when $r \mid r'$ the number $(-\th
)^{\frac{1}{p^{r}-1}}$ is a power of $(-\th )^{\frac{1}{p^{r'}-1}}$.

The formal inverse of $\exp_{C_{r}}(z)$ is the Carlitz logarithm
$\log_{C_{r}}(z)$, and as a power series in $z$, one has
$$\log_{C_{r}}(z)=\sum_{i=0}^{\infty}\frac{z^{p^{ri}}}{L_{ri}}, $$
where $$
\begin{array}{l}
   L_{r0}:=1, \\
   L_{ri}:=\prod_{j=1}^{i}({\th}-{\th}^{p^{rj}}). \\
\end{array}   $$As a function in $z$, $\log_{C_{r}}(z)$ converges for all
$z\in {\CC}_{\infty}$ with $|z|_{\infty}<
|\th|_{\infty}^{\frac{p^{r} }{p^{r}-1 } }$. It satisfies the
functional equation
$$\th \log_{C_{r}}(z)=\log_{C_{r}}(\th z)+\log_{C_{r}}(z^{p^{r}})$$
whenever the values in question are defined.

For a positive integer $n$, the $n$-th Carlitz polylogarithm
associated to $C_{r}$ is the series
\begin{equation}\label{def of polylog}
\Plog_{rn}(z):= \sum_{i=0}^{\infty} \frac{z^{p^{ri}}}{L_{ri}^{n} }
\end{equation}
 which converges $\infty$-adically
for all $z\in {\CC}_{\infty}$ with $|z|_{\infty}<
|\th|_{\infty}^{\frac{np^{r}}{p^{r}-1}}$. Its value at a particular
$z=\a\neq 0$ is called the $n$-th polylogarithm of $\a$ associated
to $C_{r}$. In transcendence theory we are interested in those
polylogarithms of $\a \in \bar{k}^{\times}$, as analogous to
classical logarithms of algebraic numbers.

\subsection{Algebraic independence of special functions} For any
positive integer $r$, we
 let
\[
\Omega_{r}(t):=(-{\th})^{\frac{-p^{r}}{p^{r}-1}}\prod_{i=1}^{\infty}
\biggl(1-\frac{t}{{\th}^{p^{ri}}} \biggr)\in
\overline{k_{\infty}}[[t]]\subseteq \mathbb{C}_{\infty}((t)).
\]
One checks that $\Omega_{r}\in \mathbb{E}$. Furthermore,
$\Omega_{r}$ satisfies the functional equation
\begin{equation}\label{fun equ of Omega r}
\Omega_{r}^{(-r)}(t)=(t-\th)\Omega_{r}(t)
\end{equation}
and its specialization at $t=\th$ gives
$\Omega_{r}(\th)=-1/\tilde{\pi}_{r}.$

By \eqref{fun equ of Omega r}, the function $\Omega_{r}$ provides a
rigid analytic trivialization of the Carlitz motive
$\mathcal{C}_{r}$ over $\FF_{p^{r}}$ that has the {\bf{GP}} property
(cf.\ Theorem \ref{refined ABP}). This is the pre-$t$-motive with
underlying space $\bar{k}(t)$ itself and $\s^{r}$ acts by
$\s^{r}f=f^{(-r)}$ for $f\in \mathcal{C}_{r}$.

For any $d\in \NN$, we let $\ell:= \lcm(1,\ldots,d)$ and
$\ell_{r}:=\frac{\ell}{r}$ for $r=1,\ldots,d$. Let
$\mathcal{C}_{r}^{(\ell_{r})}$ be the $\ell_{r}$-th derived
pre-$t$-motive over $\FF_{p^{\ell}}$ of $\mathcal{C}_{r}$, we
consider the direct sum
$M=M_{d}:=\oplus_{r=1}^{d}\mathcal{C}_{r}^{(\ell_{r})}$. By
Corollary \ref{cor of refined ABP}, $M$ also has the {\bf{GP}}
property. We note that the canonical rigid analytical trivialization
of $M$ is the diagonal matrix $\Psi\in \Mat_{d}(\mathbb{E})\cap
\GL_{d}(\mathbb{L})$ with diagonal entries
$\Omega_{1},\ldots,\Omega_{d}$.

\begin{lemma}\label{alg. ind. of Omega}
Given any positive integer $d\geq 2$, let $M=M_{d}$ be the rigid
analytically trivial pre-$t$-motive with rigid analytic
trivialization $\Psi$ defined as above. Then we have $\dim
\Gamma_{\Psi}=d$. In particular, the functions
 $\Omega_{1},\ldots,\Omega_{d} $ are algebraically independent over
$\bar{k}(t)$ and the values $\tilde{\pi}_{1},\ldots,\tilde{\pi}_{d}$
are algebraically independent over $\bar{k}$.
\end{lemma}

\begin{proof}
Suppose $\dim \Gamma_{\Psi}< d$. Since $\Psi$ is a diagonal matrix
with diagonal entries $\Omega_{1},\ldots,\Omega_{d}$, by
\eqref{Gamma Psi} we have that $\Gamma_{\Psi}\subseteq \mathrm{T}$,
where $\mathrm{T}$ is the split torus of dimension $d$ in $\GL_{d}/
{\FF_{p^{\ell}}}(t)$. We let $X_{1},\dots,X_{d}$ be the coordinates
of $\mathrm{T}$ and $\chi_{j}$ the character of $\mathrm{T}$ which
projects the $j$-th diagonal position to ${\GG}_{m}$. Note that
$\{\chi_{j} \}_{j=1}^{d}$ generates the character group of
$\mathrm{T}$. Hence $\Gamma_{\Psi}$ is the kernel of some characters
of $\mathrm{T}$, i.e., canonical generators of the defining ideal
for $\Gamma_{\Psi}$ can be of the form $X_{1}^{m_{1}}\cdots
X_{d}^{m_{d}}-1$ for some integers $m_{1},\dots,m_{d}$, not all
zero. By \eqref{Gamma Psi} we have that
$$ (\Omega_{1}^{-m_{1}}\ldots \Omega_{d}^{-m_{d}} )\otimes (\Omega_{1}^{m_{1}}\ldots \Omega_{d}^{m_{d}})=1 \in \mathbb{L}\otimes_{\bar{k}(t)}\mathbb{L}  , $$
and hence
\begin{equation}\label{beta I}
\beta:= \Omega_{1}^{m_{1}}\ldots \Omega_{d}^{m_{d}}\in
\bar{k}(t)^{\times}.
\end{equation}
We recall that $\Omega_{r}$ has zeros on $\{ {\th}^{p^{rj}}
\}_{j=1}^{\infty}$.  Since $\beta\in \bar{k}(t)^{\times}$, it has
only finitely many zeros and poles, and hence
$\ord_{t={\th}^{p^{h}}}(\beta)=0$ for $h\gg 0$. Choosing a prime
number $p'$ sufficiently large, so that the following conditions
hold:
\begin{itemize}
\item $p'> d$;
\item the order of vanishing of $\beta$ at $t={\th}^{p^{p' }}$ is zero.
\end{itemize}
These conditions imply that $m_{1}=0$ from \eqref{beta I}. Iterating
this argument, we conclude that $m_{1}=\cdots=m_{d}=0$, whence a
contradiction.
\end{proof}

\subsection{Algebraic independence of polylogarithms.}
Given $n\in {\NN}$ and ${\a}\in \bar{k}^{\times}$ with
$|{\a}|_{\infty}<|{\th}|_{\infty}^{\frac{np^{r}}{p^{r}-1}}$, we
consider the power series
$$L_{{\a},rn}(t):={\a}+\sum_{i=1}^{\infty}\frac{{\a}^{p^{ri}}}{(t-{\th}^{p^{r}})^{n} \cdots (t-{\th}^{p^{ri}})^{n} },$$
which as a function on $\mathbb{C}_{\infty}$ converges on
$|t|_{\infty}<|{\th}|_{\infty}^{p^{r}} $. We note that
$L_{{\a},rn}({\th})$ is exactly the $n$-th polylogarithm of ${\a}$
associated to $C_{r}$, i.e.,
$$  L_{{\a},rn}({\th})=\Plog_{rn}(\a). $$
Given a collection of such numbers $\alpha$, say $\alpha_1, \dots,
\alpha_m$, we define
$$\Phi_{rn}({\a}_{1}, \dots, {\a}_{m}):=\left[%
\begin{array}{cccc}
  (t-{\th})^{n} & 0 & \cdots & 0 \\
  {\a}_{1}^{(-r)}(t-{\th})^{n} & 1 & \cdots & 0 \\
  \vdots & \vdots & \ddots & \vdots \\
  {\a}_{m}^{(-r)}(t-{\th})^{n} & 0 & \cdots & 1 \\
\end{array}%
\right]\in \GL_{m+1}(\bar{k}(t))\cap \Mat_{m+1}(\bar{k}[t])
$$ and
$$
 \Psi_{rn}({\a}_{1}, \dots, {\a}_{m}):=\left[%
\begin{array}{cccc}
  \Omega_{r}^{n} & 0 & \cdots & 0 \\
  \Omega_{r}^{n}L_{{\a}_{1},rn} & 1 & \cdots & 0 \\
  \vdots & \vdots & \ddots & \vdots \\
  \Omega_{r}^{n}L_{{\a}_{m},rn} & 0 & \cdots & 1 \\
\end{array}%
\right]\in \GL_{m+1}(\mathbb{L})\cap \Mat_{m+1}(\mathbb{E}).
$$
Then one has
\begin{equation}\label{functional equation of tildePsi(rn)}
\Psi_{rn}({\a}_{1}, \dots, {\a}_{m})^{(-r)}=\Phi_{rn}({\a}_{1},
\dots, {\a}_{m}) \Psi_{rn}({\a}_{1}, \dots, {\a}_{m}) \quad
(\hbox{cf.{\cite[\S 3.1.2]{Chang-Yu}}}).
\end{equation}
Hence, $\Phi_{rn}({\a}_{1}, \dots, {\a}_{m})$ defines a rigid
analytically trivial pre-$t$-motive over $\FF_{p^{r}}$ that has the
{\bf{GP}} property.

In \cite{Chang-Yu}, we followed Papanikolas' methods to generalize
the algebraic independence of Carlitz logarithms to algebraic
independence of polylogarithms. Precisely, by \cite[Thm.
3.1]{Chang-Yu} and Theorem \ref{refined ABP} we have:

\begin{theorem}{\rm{(Chang-Yu, \cite[Thm. 3.1]{Chang-Yu})}}
\label{tr.deg of polylog} Given any positive integers $r$ and $n$,
let ${\a}_{1}, \dots, {\a}_{m}\in \bar{k}^{\times}$ satisfy
$|{\a}_{i}|_{\infty}<|{\th}|_{\infty}^{\frac{np^{r}}{p^{r}-1}}$ for
$i=1, \dots, m$.  Then
\begin{align*}
 \dim_{{\FF}_{p^{r}}({\th})}N_{rn}  &= \trdeg_{\bar{k}}
\bar{k}(\tilde{\pi}_{r}^{n},L_{{\a}_{1},rn}({\th}), \dots,
L_{{\a}_{m},rn}({\th})) \\
   &= \trdeg_{\bar{k}(t)}
\bar{k}(t)(\Omega_{r}^{n},L_{{\a}_{1},rn}, \dots, L_{{\a}_{m},rn}),
\end{align*}
where
\[
N_{rn}:={\FF}_{p^{r}}({\th})\hbox{-}
\Span\{\tilde{\pi}_{r}^{n},L_{{\a}_{1},rn}({\th}), \dots,
L_{{\a}_{m},rn}({\th})\}.
\]
\end{theorem}

\subsection{Formulas for zeta values}
\subsubsection{Euler-Carlitz relations.} In this subsection, we fix a
positive integer $r$. In \cite{Carlitz1} Carlitz introduced the
power sum
\[
\zeta_{r}(n):=\sum_{
\begin{array}{c}
  a\in {\FF}_{p^{r}}[\th] \\
  a\ \mathrm{monic} \\
\end{array}}\frac{1}{a^{n}}\in\overline{k_{\infty}}\quad \textnormal{($n$ a
positive integer)}
\]
which are the Carltz zeta values associated to
${\FF}_{p^{r}}[{\th}]$.

Writing down a $p^{r}$-adic expansion $\sum_{i}n_{ri}p^{ri}$ of $n$,
we let
\[
\Gamma_{r,n+1}:=\prod_{i=0}^{\infty}D_{ri}^{n_{ri}}.
\]
We call $\Gamma_{r,n+1}$ the {\bf{Carlitz factorials}} associated to
${\FF}_{p^{r}}[{\th}]$.  The Bernoulli-Carlitz {\bf{numbers}}
$B_{rn}$ in ${\FF}_{p}({\th})$ are given by the following expansions
from the Carlitz exponential series
$$\frac{z}{\exp_{C_{r}}(z)}=\sum_{n=0}^{\infty} \frac{B_{rn}}{\Gamma_{r,n+1}}z^{n}.$$
Carlitz proved the {\bf{Euler-Carlitz}} relations:
\begin{theorem}{\rm{(Carlitz, \cite{Carlitz1})}}
For all positive integer $n$ divisible by $p^{r}-1$, one has
\begin{equation}\label{Euler-Carlitz}
\zeta_{r}(n)=\frac{B_{rn}}{\Gamma_{r,n+1}}\tilde{\pi}_{r}^{n} .
\end{equation}
\end{theorem}
We call a positive integer $n$ $(p, r)$-{\bf{even}} if it is
divisible by $p^{r}-1$, otherwise we call $n$ $(p, r)$-{\bf{odd}}.
In particular, when $p=2$ and $r=1$, all positive integers are
${\bf{even}}$.

\subsubsection{The Anderson-Thakur formula.}
In \cite{Anderson-Thakur}, Anderson and Thakur introduced the $n$-th
tensor power of the Carlitz ${\FF}_{p^{r}}[t]$-module $C_{r}$, and
they related $\zeta_{r}(n)$ to the last coordinate of the logarithm
associated to the $n$-th tensor power of $C_{r}$ for each positive
integer~$n$. More precisely, they interpreted $\zeta_{r}(n)$ as
${\FF}_{p^{r}}(\th)$-linear combinations of $n$-th Carlitz
polylogarithms of algebraic numbers:

\begin{theorem}{\rm{(Anderson-Thakur,
\cite{Anderson-Thakur})}} Given any positive integers $r$ and $n$,
one can find a sequence $h_{rn,0}, \dots, h_{rn,l_{rn}}\in
{\FF}_{p^{r}}(\th)$, $l_{rn}<\frac{np^{r}}{p^{r}-1}$, such that the
following identity holds
\begin{equation}\label{Anderson-Thakur}
\zeta_{r}(n)=\sum_{i=0}^{l_{rn}}h_{rn,i}\Plog_{rn}(\th^{i}),
\end{equation}
where $\Plog_{rn}(z)$ is defined as in \eqref{def of polylog}. In
the special case of $n\leq p^{r}-1$,
$$\zeta_{r}(n)=\Plog_{rn}(1). $$
\end{theorem}

\begin{definition}\label{choice of i r0,..}
Given any positive integer $r$, for each $n\in {\NN}$,
$(p^{r}-1)\nmid n$, with $l_{rn}$ as given by
\eqref{Anderson-Thakur}, we fix a finite subset
$$\{ \a_{0,rn}, \dots, \a_{m_{rn},rn} \}\subseteq \{1,\th, \dots, \th^{l_{rn}} \}$$
such that
$$\{ \tilde{\pi}_{r}^{n},\mathcal{L}_{0,rn}(\th), \dots, \mathcal{L}_{m_{rn},rn}(\th) \}$$
and
$$\{ \tilde{\pi}_{r}^{n},\zeta_{r}(n), \mathcal{L}_{1,rn}(\th), \dots, \mathcal{L}_{m_{rn},rn}(\th)  \}$$
are ${\FF}_{p^{r}}({\th})$-bases for $N_{rn}$, where
$\mathcal{L}_{j,rn}(t):=L_{\a_{j},rn}(t)$ for $j=0, \dots, m_{rn}$.
This can be done because of \eqref{Anderson-Thakur} (cf. \cite[$\S$
4.1]{Chang-Yu}) and note that $m_{rn}+2$ is the dimension of
$N_{rn}$ over $\FF_{p^{r}}(\th)$. For the case of $p=2$ and $r=1$,
we put $m_{11}+1:=0$.
\end{definition}

\begin{definition}
Given any positive integers $s$ and $d$ with $d \geq 2 $. For each
$1\leq r\leq d$ we define
$$
\begin{array}{ll}
  U_{r}(s)=\{ 1 \}, & \hbox{ if }p=2 \hbox{ and }r=1; \\
  U_{r}(s)=\{1\leq n \leq s;p\nmid n,\ (p^{r}-1) \nmid n \}, & \hbox{ otherwise}.\\
\end{array} $$
For each $n\in U_{r}(s)$, we define that if $p=2$ and $r=1$, $$
\begin{array}{l}
 {\Phi}_{rn}:=(t-{\th})\in\GL_{1}(\bar{k}(t)),\\
 {\Psi}_{rn}:=\Omega_{1}\in \GL_{1}(\mathbb{L}),\\
\end{array}  $$ otherwise
\[
\begin{array}{l}
{\Phi}_{rn}:=\Phi_{rn}(\a_{0,rn}, \dots, \a_{m_{rn},rn} )\in
\GL_{(m_{rn}+2)}(\bar{k}(t)),   \\
{\Psi}_{rn}:=\Psi_{rn}(\a_{0,rn}, \dots, \a_{m_{rn},rn} )\in
\GL_{(m_{rn}+2)}(\mathbb{L}).   \\
\end{array}
\]
By Theorem \ref{tr.deg of polylog} we have that
$$\dim\Gamma_{\Psi_{rn}}= \trdeg_{\bar{k}(t)} \bar{k}(t)({\Psi}_{rn})=m_{rn}+2.$$
\end{definition}

Put ${\Phi}_{r}:= \oplus_{n\in U_{r}(s)}{\Phi}_{rn}$, then
${\Phi}_{r}$ defines a rigid analytically trivial pre-$t$-motive
${M}_{r}$ over $\FF_{p^{r}}$ with rigid analytical trivialization
${\Psi}_{r}:= \oplus_{n\in U_{r}(s)}{\Psi}_{rn}$ (cf.
\eqref{functional equation of tildePsi(rn)}). Moreover, $M_{r}$ is
of type {\bf{SV}} and by Theorem \ref{refined ABP} $M_{r}$ has the
${\bf{GP}}$ property. The main theorem of \cite{Chang-Yu} is the
following:

\begin{theorem}{\rm{(Chang-Yu}, \cite[Thm.~4.5]{Chang-Yu})}\label{Chang-Yu}
For any positive integers $s$ and $r$, the Galois group
$\Gamma_{M_{r}}$ over $\FF_{p^{r}}(t)$ is full, i.e.,
$$\dim\Gamma_{{M}_{r}} =1+\sum_{n\in U_{r}(s)}(m_{rn}+1).$$
\end{theorem}

\subsection{Proof of Theorem \ref{main thm introduction}}
Given any integer $d\geq 2$, we put $\ell:= \lcm (1,\ldots,d)$ and
$\ell_{r}:=\frac{\ell}{r}$ for $r=1,\ldots,d$. For each $1\leq r\leq
d$ let ${\bf{M}}_{r}:=M_{r}^{(\ell_{r})} $ be the $\ell_{r}$-th
derived pre-$t$-motive over $\FF_{p^{\ell}}$ of $M_{r}$ defined as
above, then it is still of type {\bf{SV}}. By
Proposition~\ref{direct sum of GP motives} each ${\bf{M}}_{r}$ has
the {\bf{GP}} property and by Theorem \ref{Chang-Yu} its Galois
group $\Gamma_{\mathbf{M}_{r}}$ is full. Further, for each $1\leq r
\leq d$ any diagonal of $\mathbf{M}_{r}$ has canonical rigid
analytical trivialization given by $\Omega_{r}^{n}$ for some $n\in
U_{r}(s)$ and hence its Galois group is $\GG_{m}$ because
$\Omega_{r}$ is transcendental over $\bar{k}(t)$.

Since by Lemma \ref{alg. ind. of Omega} the functions
$\Omega_{1},\ldots,\Omega_{d}$ are algebraically independent over
$\bar{k}(t)$, particularly the Galois group of $N_{i}\oplus N_{j}$
is a two-dimensional torus over $\FF_{p^{\ell}}(t)$ for any diagonal
$N_{i}$ (resp. $N_{j}$) of $\mathbf{M}_{i}$ (resp. $\mathbf{M}_{j}$)
with $i\neq j$, $1\leq i,j\leq d$. Put ${\bf{M}}:=\oplus_{r=1}^{d}
{\bf{M}}_{r}$, then by  Corollary \ref{cor of refined ABP}
${\bf{M}}$ has the {\bf{GP}} property. Applying Theorem \ref{main
thm} to this situation, we obtain the explicit dimension of
$\Gamma_{\mathbf{M}}$:
\begin{theorem}\label{main thm of last sec}
Given any positive integers $s$ and $d$ with $d\geq 2$, let
${\bf{M}}$ be defined as above. Then the Galois group
$\Gamma_{{\bf{M}}}$ is full, i.e.,
$$\dim \Gamma_{\bf{M}}=d+ \sum_{r=1}^{d}(m_{rn}+1  ) .$$

\end{theorem}

 As a consequence, we completely determine all the algebraic relations among the
families of Carlitz zeta values:

\begin{corollary}\label{cor of the main thm}
Given any positive integers $d$ and $s$, the transcendence degree of
the field
$$ \bar{k}\bigl(\cup_{r=1}^{d}\{
\tilde{\pi}_{r},\zeta_{r}(1), \dots, \zeta_{r}(s)\}\bigr)$$ over
$\bar{k}$ is
$$ \sum_{r=1}^{d}\biggl(s-\biggl\lfloor\frac{s}{p}  \biggr\rfloor
- \biggl\lfloor\frac{s}{p^{r}-1} \biggr\rfloor + \biggl\lfloor
\frac{s}{p(p^{r}-1)} \biggr\rfloor +1 \biggr).$$
\end{corollary}

\begin{proof}
We may assume $d\geq 2$ since
 the case $d=1$ is already given in \cite[Cor. 4.6]{Chang-Yu}. For $1\leq r \leq d$, let
$$
\begin{array}{l}
 V_{1}(s):=\varnothing, \hbox{ if }p=2;  \\
 V_{r}(s):=U_{r}(s),\hbox{ otherwise.}  \\
\end{array}  $$
Since $\mathbf{M}$ has the {\bf{GP}} property,
 by Theorem \ref{main thm of last sec} we see that the following elements
\begin{equation}\label{trans. basis}
\{ \Omega_{1}({\th}), \dots, \Omega_{d}({\th})\} \cup
\Bigl(\cup_{r=1}^{d}\cup_{n\in V_{r}(s)} \{ \mathcal{L}_{0,rn}(\th),
\dots, \mathcal{L}_{m_{rn},rn}(\th) \} \Bigr)
\end{equation}
are algebraically independent over $\bar{k}$.  In particular, by
Definition \ref{choice of i r0,..} we have that
$$
\{ \tilde{\pi}_{1}, \dots, \tilde{\pi}_{d} \} \cup \Bigl(
\cup_{r=1}^{d}\cup_{n\in V_{r}(s)} \{ \zeta_{r}(n) \} \Bigr)
$$
is an algebraically independent set over $\bar{k}$. Counting the
cardinality of  $V_{r}(s)$ for each $1\leq r \leq d$, we complete
the proof.
\end{proof}

\bibliographystyle{amsplain}

\end{document}